\documentclass[a4paper,12pt]{article}
\usepackage{epsfig,amssymb,amsmath,amsthm,fullpage}
\newtheorem{thm}{Theorem}

\newtheorem{lem}[thm]{Lemma}
\newtheorem{propn}[thm]{Proposition}
\newtheorem{rem}{Remark}

\numberwithin{equation}{section}

\title{Heat kernel fluctuations for a resistance form with non-uniform volume growth}
\author{D.A. Croydon\footnote{Dept of Statistics,
University of Warwick, Coventry CV4 7AL, UK;
{d.a.croydon@warwick.ac.uk}.}}
\date{2 February 2007}

\begin{document}
\maketitle

\begin{abstract}
In this article, we consider the problem of estimating the heat kernel on measure-metric spaces equipped with a resistance form. Such spaces admit a corresponding resistance metric that reflects the conductivity properties of the set. In this situation, it has been proved that when there is uniform polynomial volume growth with respect to the resistance metric the behaviour of the on-diagonal part of the heat kernel is completely determined by this rate of volume growth. However, recent results have shown that for certain random fractal sets, there are global and local (point-wise) fluctuations in the volume as $r\rightarrow 0$ and so these uniform results do not apply. Motivated by these examples, we present global and local on-diagonal heat kernel estimates when the volume growth is not uniform, and demonstrate that when the volume fluctuations are non-trivial, there will be non-trivial fluctuations of the same order (up to exponents) in the short-time heat kernel asymptotics. We also provide bounds for the off-diagonal part of the heat kernel. These results apply to deterministic and random self-similar fractals, and metric space dendrites (the topological analogues of graph trees).
\end{abstract}

\section{Introduction}

We start by introducing the general framework and notation that we will use throughout the article. Let $(X,d)$ be a locally compact, separable, connected metric space, and $\mu$ be a non-negative Borel measure on $X$, finite on compact sets and strictly positive on non-empty open sets. Assume that $(\mathcal{E},\mathcal{F})$ is a local, regular Dirichlet form on $L^2(X,\mu)$ and that there exists $\mathcal{F}'\supseteq\mathcal{F}$ such that $(\mathcal{E},\mathcal{F}')$ is a resistance form on $X$, (for an introduction to resistance forms, see \cite{Kigami}, Section 2.3). Resistance forms arise naturally from self-similar fractals and metric space dendrites, see \cite{Kigamidendrite} and \cite{Kigami} for examples and a precise definition.

Define the resistance function $R$ associated with $(\mathcal{E},\mathcal{F})$ by
\begin{equation}\label{resfunc}
R(A,B)^{-1}:=\inf\{\mathcal{E}(f,f):\:f\in\mathcal{F},f|_A=1,f|_B=0\},
\end{equation}
for disjoint subsets $A,B$ of $X$. If we set $R(x,y)=R(\{x\},\{y\})$, for $x\neq y$, and $R(x,x)=0$, then using the fact that $(\mathcal{E},\mathcal{F})$ is contained in a resistance form, it may be shown that the function $R:X\times X\rightarrow [0,\infty)$ is a metric on $X$. This metric is called the {\it resistance metric}, and we shall assume that the topology induced by $R$ is compatible with the topology induced by $d$. Note that, in the electrical network interpretation of a quadratic form, the right hand side of (\ref{resfunc}) is precisely the effective conductivity between the sets $A$ and $B$. Thus, the resistance function represents the effective resistance between sets. We shall denote by $B(x,r)$ the connected component of the resistance ball of radius $r$ around $x$ containing $x$.

Given a Dirichlet form, there is a natural way to associate it with a non-negative self-adjoint operator, $-\mathcal{L}$, which has a domain dense in $L^2(X,\mu)$ and satisfies
$$\mathcal{E}(f,g)=-\int_X f\mathcal{L}g d\mu,\hspace{20pt}\forall f\in \mathcal{F}, g\in \mathcal{D}(\mathcal{L}).$$
Through this association, we may define a related Markov process, $(X_t)_{t\geq0}$, with semi-group given by $P_t:=e^{t\mathcal{L}}$. Under the assumptions we have made so far, it may be proved that, because our Dirichlet form is local, our process is a diffusion and also, when $X$ is compact, $(P_t)_{t\geq 0}$ is Feller. Furthermore, we prove in Section \ref{existence} there exists a version of the transition density $p_t$, for each $t>0$, and it is this that will be the object of interest in this article. Apart from in Section \ref{existence}, we shall refer to it as the heat kernel or transition density interchangeably.

In the resistance form setting, it has been established that knowledge of the volume growth with respect to the resistance metric is an important factor in determining the behaviour of the heat kernel. One widely applicable way of describing volume growth is the idea of volume doubling. To introduce this, suppose that we have a strictly increasing function $V$, with $V(0)=0$, that satisfies the doubling condition:
\begin{equation}\label{doubling}
V(2r)\leq C_u V(r).
\end{equation}
We say our measure-metric space, $X$, has {\it uniform volume doubling} if we can find a function $V$ satisfying the above properties and also $C_1V(r)\leq V(x,r)\leq C_2 V(r)$, for every $x\in X$ and $r\in[0,R_X)$, where $V(x,r):=\mu(B(x,r))$, and $R_X$ is the diameter of $X$ with respect to the resistance metric, which may be infinite. This uniform volume growth condition includes any space with uniform polynomial volume growth, but excludes exponential growth.

In \cite{Kumagai}, for a measure-metric space satisfying the conditions of this article, Kumagai proves that uniform volume doubling implies that there exists a constant $T_X>0$ depending only on $(X,R)$ such that the following upper bound on the heat kernel holds: for $x,y\in X$, $t\in(0,T_X]$,
\begin{equation}\label{kumagairesult}
p_t(x,y)\leq \frac{C_3h^{-1}(t)}{t}e^{-\frac{R(x,y)}{C_3V^{-1}(t/R(x,y))}},
\end{equation}
where $h(r):=rV(r)$ occurs as a time scale function. It is also demonstrated that a near diagonal lower bound of the form
\begin{equation}\label{kumagairesultlower}
p_t(x,y)\geq \frac{C_4h^{-1}(t)}{t},\hspace{20pt}\mbox{for $h(C_4R(x,y))\leq t$},
\end{equation}
holds for $t\in(0,T_X]$. In particular, uniform volume doubling in the resistance metric determines that the on-diagonal part of the heat kernel is given up to constant multiples by $h^{-1}(t)/t$.

A major motivation for investigating the properties of the heat kernel on measure-metric spaces equipped with a resistance form is provided by fractal spaces. On many regular self-similar fractals, such as the Sierpinski gasket, the most convenient way to define a Laplacian on the set is via the construction of a self-similar resistance form, see \cite{Barlow}, \cite{Kigami}. Furthermore, the high degree of symmetry of these sets allows it to be deduced that uniform volume doubling holds, and thus, the results of \cite{Kumagai} immediately apply. However, this uniformity of volume growth has been shown not to be the case in the random fractal setting. In \cite{HamJon}, Hambly and Jones prove that for a class of random recursive fractals we can do no better than to bound the measures of balls by
$$C_{5} V(r)(\ln r^{-1})^{-a_1}\leq V(x,r)\leq C_6 V(r)(\ln r^{-1})^{a_2},$$
where $V(r)=r^\alpha$ and $C_{5}, C_6, a_1, a_2$ are strictly positive constants. Note that, although in \cite{HamJon} the volume growth is presented in terms of the original metric, it is straightforward to show that the same kinds of fluctuations occur when we consider the resistance metric balls. This uneveness, caused by the random construction mechanism, means that the uniform results do not apply. In fact, there is not even local (pointwise) volume doubling in this example. In \cite{HamKum}, Hambly and Kumagai show that the best possible upper and lower bounds for the on-diagonal part of the heat kernel on a random Sierpinski gasket are not asymptotically multiples of each other and also exhibit logarithmic fluctuations.

The main purpose of this article is to approach the problem of having non-uniform volume growth more generally. We make no assumptions on the specific structure of our measure-metric space and place only weak conditions on the fluctuations we use in the volume growth condition, see Section \ref{framework}. The argument we use follows closely that of Kumagai, \cite{Kumagai}, for the case of uniform volume doubling, although more work is required to deal with the fluctuations. As one would expect, by considering the problem in such generality, the results we get are not as sharp as those obtained in specific cases. However, we demonstrate that the loss of accuracy can only be the exponents of the correction terms. We shall discuss this further in Section \ref{examples} for some particular examples. The advantage of taking this approach is that we are able to deduce widely applicable bounds, and a particularly nice feature of the results we obtain is that the correction terms of the heat kernel bounds depend on the correction terms of the measure bounds in simple, explicit ways. For example, if we have logarithmic corrections to the measure, our results imply that there are no worse than logarithmic corrections to the heat kernel.

The estimation of heat kernels has of course been of interest in various other settings. Aronson, \cite{Aronson}, derived upper and lower bounds on the heat kernel for an elliptic operator in $\mathbb{R}^n$ and since then, the behaviour of the heat kernel for elliptic operators on Riemannian manifolds has been studied extensively, see \cite{Grig2} for an introduction to this area. Closely related to this, through discretisation techniques, is the estimation of heat kernels on graphs, where for these spaces, heat kernels are most easily thought of as the transition densities of the associated simple random walks. By considering a graph to be an electrical network, where each edge has resistance one, then we can define the resistance metric by taking $R(x,y)$ to be the effective resistance between vertices $x$ and $y$. In this case, if we have uniform volume doubling in the resistance metric, then suitable modifications of the results obtained by Kumagai for resistance forms allow it to be deduced that the on-diagonal part of the (discrete time) heat kernel behaves like $h^{-1}(n)/n$, for large $n$.

For a graph, it is not always straightforward to calculate the resistance between points, and the more natural distance to use is the shortest path length metric, $d$. Furthermore, the volume growth with respect to $d$ can sometimes be very different to that with respect to $R$. For example, on the integer lattice $\mathbb{Z}^2$, there is uniform volume doubling in the metric $d$, as the volume grows like $r^2$, whereas in the resistance metric, the volume grows exponentially in $r$. Because the distance $d$ is easier to calculate, there has been a great deal of effort put into establishing heat kernel estimates using knowledge of the volume growth with respect to $d$. As shown in \cite{BCG}, the information contained by the volume growth in $d$ is insufficient to characterise the heat kernel behaviour, and a range of outcomes is possible. However, for fractal-type graphs the resistance and shortest path metrics are often more closely linked, with some kind of power law between the two holding. In fact, when the volume growth is polynomial (in $d$), in \cite{BCK} it is shown that double-sided (sub-Gaussian) heat kernel estimates hold if and only if such a connection holds. The relationship between $d$ and $R$ is most obvious in the case of graph trees, where the two are in fact identical. Consequently, it is to fractal-type graphs and graph trees that the resistance form results are most easily adapted.

By analogy with the random recursive fractals of \cite{HamJon} and \cite{HamKum}, one might expect that the kind of uniform volume growth that holds for many fractal-type graphs does not hold when random variants are considered. In fact, this has already been proved in the case of the incipient infinite cluster of critical percolation on the binary tree, where local fluctuations of order $\ln\ln r$ about a leading order $r^2$ term occur in $V(x,r)$, see \cite{BarKum}. Note that, since this structure is a graph tree, this is the volume growth with respect to the resistance metric. In the same article, it was shown that these measure fluctuations lead to fluctuations of log-logarithmic order in the heat kernel, which mirrors the results of this article. As in the uniform volume doubling case, it should be a matter of making simple modifications to the techniques used here for resistance forms to exhibit fluctuation results for graphs more generally.

Of greater relevance to our situation are dendrites, which are the topological analogues of graph trees, with their defining properties being that they are arcwise-connected and contain no subset homeomorphic to the circle. For these sets, it was shown by Kigami, \cite{Kigamidendrite}, that any shortest path (additive along paths) metric, $d$, is in fact a resistance metric for some resistance form. Thus for these sets the volume growth in the original metric, $d$, and in the resistance metric, $R$, coincides. Using the simpler structure of these spaces, under uniform volume doubling, it is also possible to obtain a lower bound for the heat kernel of the same form as (\ref{kumagairesult}) with a different constant, see \cite{Kumagai}. Although the assumptions that make a space a dendrite are restrictive, there are many important examples, including the continuum random tree of Aldous, see \cite{Aldous2}. This is a random dendrite that arises naturally as the scaling limit of various families of random graph trees, and demonstrates measure fluctuations of the kind considered here, \cite{Croydoncrt}. For further discussion, see Section \ref{examples}.

The format of the article is as follows. In Section \ref{framework} we introduce the volume fluctuations that are considered in this article. In Section \ref{ondiagstate}, we state bounds for the on-diagonal part of the heat kernel, which depend only on the volume growth condition given at (\ref{volumegrowthcondition}). Furthermore, we show that when there actually are fluctuations in the measure of the kind described at (\ref{volumegrowthcondition}), there will also be spatial fluctuations in the heat kernel. In Section \ref{offdiagstate}, we state the bounds for the off-diagonal part of the heat kernel. To obtain the full bound we assume a chaining condition, which is also defined in this section. Section \ref{existence} is where the existence of the heat kernel is checked. Our main results are then proved in Sections \ref{ondiagsec} and \ref{offdiagonal}. Following this, in Section \ref{local}, we discuss the effect of having local fluctuations in the measure. Finally, in Section \ref{examples}, we discuss polynomial and logarithmic corrections to the volume growth function $V(r)=r^\alpha$. We shall also compare our on-diagonal results to those already established for random Sierpinski gaskets, and preview results for the continuum random tree. Constants of the form $c_.$ take values in $(0,\infty)$ and may take different values in different results.

\section{Volume fluctuations}\label{framework}

In this section, we make precise the volume growth condition that we shall presuppose for the remainder of the article. First, as in the introduction, let $V$ be a strictly increasing function, with $V(0)=0$, that satisfies the doubling condition of (\ref{doubling}). We will define $\beta_u:=\ln C_u /\ln 2$ to be the upper volume growth exponent, and continue to use the notation $h(r):=rV(r)$. Secondly, we assume that there exist functions $f_l, f_u:[0,R_X)\rightarrow [0,\infty]$ such that
\begin{equation}\label{volumegrowthcondition}
f_l(r)V(r)\leq V(x,r) \leq f_u(r)V(r),\hspace{20pt}\forall x\in X, r\in [0,R_X),
\end{equation}
where $V(x,r):=\mu(B(x,r))$, and $B(x,r)$ is the connected component of the resistance ball containing $x$, as in the introduction. Typically, we are considering the case when the volume growth is primarily determined by $V$ and the functions $f_l$ and $f_u$ are lower order fluctuations. This is formalised in the conditions given below on $f_l$ and $f_u$, although it is possibly more enlightening to refer to the examples in Section \ref{examples}. We will use the notation $V_l(r)$, $V_u(r)$ to represent $f_l(r)V(r)$, $f_u(r)V(r)$ respectively. Similarly, we define $h_l(r)=rV_l(r)$ and $h_u(r):=rV_u(r)$. The restrictions we make on $f_l$ and $f_u$ are the following:
\newcounter{listcount}
\begin{list}{(\roman{listcount})}{\usecounter{listcount} \setlength{\rightmargin}{\leftmargin}}
\item $f_l(r)^{-1},f_u(r)=O(r^{-\varepsilon})$, as $r\rightarrow 0$, for some $\varepsilon>0$.
\item $f_l(r)$ is increasing, $f_u(r)$ is decreasing.
\item $f_l(r)^{1/b},f_u(r)^{-1/b}$ are concave on $[0,r_0]$, for some $b, r_0>0$.
\end{list}
Here, $b$ and $\varepsilon$ are constants upon which we will place upper bounds in Sections \ref{ondiagstate} and \ref{offdiagstate}. Without loss of generality, by rescaling if necessary, we can assume further that $f_l\leq 1$ and $f_u\geq 1$. It turns out that the ratio of $f_l$ to $f_u$ is particularly useful in stating our main results, and we shall notate it as follows
$$g(r):=\frac{f_l(r)}{f_u(r)}.$$
By the assumptions on $f_l$ and $f_u$, we have that $g$ is increasing, $\leq1$ and $g(r)^{-1}=O(r^{-2\varepsilon})$ as $r\rightarrow 0$.

\section{Statement of on-diagonal results}\label{ondiagstate}

We are now ready to present our first results, which explain the behaviour of the on-diagonal part of the heat kernel when the volume growth of the previous section is assumed. The upper bound on the constants $b$ and $\varepsilon$, which appear in the conditions of the volume fluctuation functions $f_l$ and $f_u$, that we require is the following:
\begin{equation}\label{betaepscondition1}
b,\:\varepsilon<\frac{1}{4(2+\beta_u)}.
\end{equation}
We also define $\theta_1$ to be a constant that satisfies
\begin{equation}\label{theta1cond}
\theta_1>\frac{(3+2b+2\beta_u)(2+\beta_u)}{1-2b(3+2b+2\beta_u)}.
\end{equation}
This is an exponent that arises in the course of establishing the following on-diagonal heat kernel bounds, which are proved in Section \ref{ondiagsec} as Propositions \ref{ondiagupper} and \ref{ondiaglowerhk}.

\begin{thm}
There exist constants $t_0>0$ and $c_1, c_2, c_3$ such that
$$c_1\frac{h^{-1}(t)}{t}g(h^{-1}(t))^{\theta_1}\leq p_t(x,x)\leq c_2 \frac{h_l^{-1}(t)}{t}\leq c_3\frac{h^{-1}(t)}{t}f_l(h^{-1}(t))^{-1},$$
for all $x\in X, t\in (0,t_0)$. If $R_X=\infty$ then we may take $t_0=\infty$, otherwise $t_0$ is finite.
\end{thm}
\begin{rem} The bound on the right hand side of this theorem is in general strictly worse than the bound involving $h_l^{-1}(t)$. However, we include it here because it demonstrates clearly that the type of fluctuations in the heat kernel are no worse than those in the measure.
\end{rem}

The next result shows that, if there actually are asymptotic fluctuations in the measure of the order of $f_l$ and $f_u$, then there will be spatial fluctuations in the heat kernel asymptotics.

\begin{thm} \label{fluctres}
If
\begin{equation}\label{infcond}
0<\liminf_{r\rightarrow0}\inf_{x\in X} \frac{V(x,r)}{V_l(r)}\leq \limsup_{r\rightarrow0}\inf_{x\in X} \frac{V(x,r)}{V_l(r)}<\infty,
\end{equation}
and
\begin{equation}\label{supcond}
0<\liminf_{r\rightarrow0}\sup_{x\in X} \frac{V(x,r)}{V_u(r)}\leq \limsup_{r\rightarrow0}\sup_{x\in X} \frac{V(x,r)}{V_u(r)}<\infty;
\end{equation}
then
\begin{equation}\label{infres}
0<\liminf_{t\rightarrow0}\inf_{x\in X} \frac{tp_t(x,x)}{h^{-1}(t)g(h^{-1}(t))^{\theta_1}},\hspace{10pt} \limsup_{t\rightarrow0}\inf_{x\in X} \frac{tp_t(x,x)}{h_u^{-1}(t)}<\infty,
\end{equation}
and
\begin{equation}\label{supres}
0<\liminf_{t\rightarrow0}\sup_{x\in X} \frac{tp_t(x,x)}{h_l^{-1}(t)}\leq\limsup_{t\rightarrow0}\sup_{x\in X} \frac{tp_t(x,x)}{h_l^{-1}(t)}<\infty.
\end{equation}
\end{thm}
\begin{rem}
Note that we have non-trivial fluctuations in the measure if and only if $V_u(r)/V_l(r)\rightarrow \infty$ as $r\rightarrow 0$. This is equivalent to $h_l^{-1}(t)/h_u^{-1}(t)\rightarrow \infty$ as $t\rightarrow0$, which implies that there are non-trivial fluctuations in the heat kernel over space.
\end{rem}

\section{Statement of off-diagonal results}\label{offdiagstate}

To obtain the off-diagonal heat kernel bounds we shall assume again that we have volume growth bounded as at (\ref{volumegrowthcondition}). We also need two extra conditions and we introduce those now. We shall be slightly stricter about how the function $V(r)$ behaves for small $r$. We shall assume that there exist constants $R_X'>0$, $C_l>1$ such that
\begin{equation}\label{volumegrowthlower}
C_lV(r)\leq V(2r),\hspace{20pt}\forall r\leq R_X',
\end{equation}
and define $\beta_l:=\ln C_l/\ln 2$, the lower growth exponent. Comparing this to equation (\ref{doubling}) means that we must have $\beta_l\leq \beta_u$. This condition ensures that $V$ increases suitably quickly near 0, and is sometimes referred to in the literature as the anti-doubling property. We shall also tighten the conditions on $b$ and $\varepsilon$ to
\begin{equation}\label{betaepscondition2}
b,\:\varepsilon<\frac{\beta_l}{8(2+\beta_u)^2},
\end{equation}
and define $\theta_1$, $\theta_2$ and $\theta_3$ to be exponents satisfying
\begin{equation}\label{theta1cond2}
\frac{\beta_l}{2b}\wedge\frac{\beta_l}{2\varepsilon}>\theta_1>\frac{(3+2b+2\beta_u)(2+\beta_u)}{1-2b(3+2b+2\beta_u)},
\end{equation}
\begin{equation}\label{theta2cond}
\theta_2>\frac{\theta_1(1+\beta_l)}{\beta_l-2b\theta_1},
\end{equation}
$$\theta_3=(3+2b+2\beta_u)(1+2\beta_l^{-1}).$$
Note that our assumptions on $b$ and $\varepsilon$ at (\ref{betaepscondition2}) mean that it is indeed possible to choose $\theta_1$ satisfying (\ref{theta1cond2}). Furthermore, we can choose $\theta_1$ that is consistent with (\ref{theta1cond}) and (\ref{theta1cond2}), we have merely added an upper bound.

Under these assumptions, we are able to deduce the following result for the off-diagonal parts of the heat kernel. It is proved in Section \ref{offdiagonal} as Propositions \ref{offdiagupperprop} and \ref{offdiaglowerprop}. In the statement of the result we use the chaining condition $(CC)$, which is defined as follows: there exists a constant $c_1$ such that for all $x,y\in X$ and all $n\in \mathbb{N}$, there exists $\{x_0,x_1,\dots,x_n\}\subseteq X$ with $x_0=x$, $x_n=y$ such that
$$R(x_{i-1},x_i)\leq c_1 \frac{R(x,y)}{n},\hspace{20pt}\forall 1\leq i \leq n.$$When this assumption holds, the following bounds show that the exponential decay away from the diagonal differs from the uniform case by a factor that is of an order no greater than the measure fluctuations (up to exponents).

\begin{thm} \label{fullbounds}
There exist constants $t_0>0$ and $c_1, c_2$ such that
$$p_t(x,y)\leq c_1\frac{h^{-1}(t)}{t}f_l(h^{-1}(t))^{-1}e^{-c_{2} \frac{R}{V^{-1}(t/R)}g(V^{-1}(t/R))^{\theta_3}},$$
for all $x,y\in X$, $t\in(0,t_0)$, where $R:=R(x,y)$.

Furthermore, if $(CC)$ holds, then there exist constants $t_1>0$ and $c_3, c_4$ such that
$$p_t(x,y)\geq c_3\frac{h^{-1}(t)}{t}g(h^{-1}(t))^{\theta_1}e^{-c_{4} \frac{R}{V^{-1}(t/R)}g(V^{-1}(t/R))^{-\theta_2}},$$
for all $x,y\in X$, $t\in(0,t_1)$, where $R:=R(x,y)$.

Note that, if $R_X'=\infty$ then we may take $t_0=t_1=\infty$, otherwise $t_0$ and $t_1$ are finite.
\end{thm}
\begin{rem}
We note that the results of Sections \ref{ondiagstate} and \ref{offdiagstate} reduce to those obtained by Kumagai in \cite{Kumagai} when $f_l$ is bounded away from 0 and $f_u$ is bounded above by a finite constant. The extension of the near-diagonal lower bound of (\ref{kumagairesultlower}) is proved in Lemma \ref{nle}.
\end{rem}

\begin{rem}
Choosing $\theta_1$ and $\theta_2$ closer to the lower bound will give tighter bounds asymptotically.
\end{rem}

\begin{rem}
The chaining condition is not necessary to obtain the off-diagonal upper bound. However, as is remarked in \cite{Kumagai}, Section 5, by Kumagai, even for the case of uniform volume doubling, the bound is not optimal in general when $(CC)$ does not hold, which is often. We note that the chaining condition holds most obviously when $X$ is a dendrite.
\end{rem}

\section{Existence of the transition density}\label{existence}

In this section, we prove the existence of a transition density for $(P_t)_{t>0}$, using a result appearing in \cite{Grig}, by Grigor'yan. The key step is establishing the ultracontractivity of the semi-group in our setting. We shall start by defining this property and the other standard terms that will be used in this section.

A semi-group $(P_t)_{t>0}$ is said to be {\it ultracontractive}\index{ultracontractive} if there exists a positive, decreasing function $\gamma(t)$ on $(0,\infty)$ such that
\begin{equation}\label{contract}
\|P_t f\|_2\leq \gamma(t) \|f\|_1,\hspace{20pt}\forall f\in L^1(X,\mu)\cap L^2(X,\mu).
\end{equation}
This property is particularly appealing for a semi-group, and as we explain below, it immediately guarantees the existence of a transition density for $(\mathcal{E},\mathcal{F})$. In the resistance form setting, we show in Proposition \ref{mainyo} that the only condition needed to deduce ultracontractivity is a suitable uniform lower bound on the volume of resistance balls.

A family $(p_t)_{t>0}$ of $\mu\times\mu$-measurable functions on $X\times X$ is called a (symmetric) {\it transition density}\index{transition density} of the semi-group $(P_t)_{t>0}$ (alternatively, of the form $(\mathcal{E},\mathcal{F})$) if there exists $X'\subseteq X$ with $\mu (X\backslash X')=0$ such that, for any bounded measurable function $f$,
$$P_t f(x)=\int_X p_t(x,y) f(y)\mu(dy),\hspace{20pt} \forall x\in X', t>0,$$
$$p_t(x,y)=p_t(y,x)\hspace{20pt}\forall x,y\in X, t>0,$$
and
$$p_{s+t}(x,y)=\int_X p_s(x,z)p_t(z,y)\mu(dz),\hspace{20pt}\forall x,y\in X, s,t>0.$$
Similarly, a family $(\tilde{p}_t)_{t>0}$ of $\mu\times\mu$-measurable functions on $X\times X$ is called a {\it heat kernel}\index{heat kernel} of $(P_t)_{t>0}$ if $\tilde{p}_t$ is an integral kernel of $P_t$ for each $t>0$. Clearly, this only defines a heat kernel up to a $\mu$-null set. The extra conditions on the transition density mean that it is defined everywhere in $X$ and is also a heat kernel. Consequently, for an arbitrary heat kernel our results only apply $\mu$-almost everywhere.

Before we prove the existence of a transition density for $(\mathcal{E},\mathcal{F})$, we state the crucial lemma that we will apply, the proof of which relies on the Riesz representation theorem. It should be noted that the argument we use for our main result, Proposition \ref{mainyo}, is standard, and is similar to the proof of the heat kernel upper bound proved in \cite{Kumagai}, Proposition 4.1. In the proof, we will utilise the following observation that is straightforward to derive from the definition of $R$. In particular, we have that
\begin{equation}\label{resistance}
|f(x)-f(y)|^2\leq R(x,y)\mathcal{E}(f,f),\hspace{20pt}\forall x,y\in X, f\in\mathcal{F}.
\end{equation}
This inequality, together with the assumption that the topologies induced by $R$ and $d$ are compatible, means that $\mathcal{F}\subseteq C(X)$, where $C(X)$ is the space of continuous functions on $(X,d)$.

{\lem \label{inctd} {\rm (\cite{Grig}, Lemma 8.1)} If the semi-group $(P_t)_{t>0}$ is ultracontractive, then it admits a transition density.}

{\propn \label{mainyo} There exists a transition density $(p_t)_{t>0}$ for $(P_t)_{t>0}$, and moreover, for each $t>0$, $p_t(x,y)$ is jointly continuous in $x$ and $y$.}
\begin{proof} By rescaling, to demonstrate that $(P_t)_{t>0}$ is ultracontractive, it is sufficient to check that (\ref{contract}) holds for every $f\in L^1(X,\mu)\cap L^2(X,\mu)$ with $||f||_1=1$. Consequently, we take $f$ to be a function satisfying these conditions, and we denote $f_t:=P_tf$. By standard semi-group theory, we note that $f_t\in\mathcal{D}(\mathcal{L})\subseteq\mathcal{F}$ for every $t>0$, where $\mathcal{D}(\mathcal{L})$ is the domain of the generator of $(P_t)_{t> 0}$. Now observe that we must have, for every $x\in X$, $r,t>0$,
$$\int_{B(x,r)}|f_t(y)|\mu(dy)\leq ||f_t||_1\leq ||f||_1=1.$$
Hence, there must exist a $y\in B(x,r)$ such that $|f_t(y)|\leq V(x,r)^{-1}\leq V_l(r)^{-1}$, where we apply the volume bound of (\ref{volumegrowthcondition}) for the second inequality. Combining this result with the inequality that was stated at (\ref{resistance}), it is possible to deduce that
\begin{eqnarray*}
\frac{1}{2}|f_t(x)|^2&\leq &|f_t(y)|^2+|f_t(x)-f_t(y)|^2\\
&\leq& V_l(r)^{-2}+r\mathcal{E}(f_t,f_t).
\end{eqnarray*}
We now define $\psi(t):=||f_t||_2^2$, which is a positive decreasing function. The above inequality allows us to write
\begin{eqnarray*}
\psi(t/2)&=&\int_X f_{t/2}(x)f_{t/2}(x)\mu(dx)\\
&\leq&\int_X |f(x)f_{t}(x)|\mu(dx)\\
&\leq& 2^{1/2}(V_l(r)^{-2}+r\mathcal{E}(f_t,f_t))^{1/2},
\end{eqnarray*}
where for the final inequality we use the fact that $||f||_1=1$. Applying established results for semi-groups, we have that $\psi'(t)=-2\mathcal{E}(f_t,f_t)$. Thus, the above inequality may be rearranged to give
$$\psi'(t)\leq\frac {2V_l(r)^{-2}-\psi(t)^2}{r},$$
where we also apply the fact that $\psi(t)\leq\psi(t/2)$. By following the proof of \cite{Kumagai}, Proposition 4.1, we are able to deduce from this differential inequality the existence of constants $c_{1}, t_0>0$ such that
\begin{equation}\label{trandensupper}
\psi(t)\leq c_{1}\frac{h_l^{-1}(t)}{t},\hspace{20pt}\forall t\in(0,t_0),
\end{equation}
which implies that in (\ref{contract}) we may take $\gamma(t)=(c_{1}h_l^{-1}(t)/t)^{1/2}$ for $t\in (0,t_0)$. Hence, $(P_t)_{t>0}$ is ultracontractive, and so, by Lemma \ref{inctd}, it admits a transition density $(p_t)_{t>0}$.

To prove the continuity of $p_t$ for each $t>0$, we first observe that $p_t(x,\cdot)=P_{t/2}p_{t/2}(x,\cdot)$. This implies that $p_{t}(x,\cdot)\in\mathcal{D}(\mathcal{L})\subseteq\mathcal{F}$, and in particular we must have $\mathcal{E}(p_t(x,\cdot),p_t(x,\cdot))<\infty$. Consequently, we can apply the inequality at (\ref{resistance}) and the symmetry of the transition density to deduce the desired continuity result.
\end{proof}

\section{Proof of on-diagonal heat kernel bounds}\label{ondiagsec}

In this section we determine bounds for the on-diagonal part of the heat kernel. We start with the proof of the upper bound. As is often the case, this is relatively straightforward to obtain. It is the lower bound which requires more work and the rest of the section is dedicated to this. A result of interest in its own right is Proposition \ref{expectedexittimes}, where we present bounds for the expected time to exit a ball.

\begin{propn} \label{ondiagupper}
There exist constants $t_0>0$ and $c_{1}$ such that
$$p_t(x,x)\leq c_{1}\frac{h_l^{-1}(t)}{t}\leq c_{1}\frac{h^{-1}(t)}{t}f_l(h^{-1}(t))^{-1},\hspace{20pt}\forall x\in X, t\in(0,t_0).$$
If $R_X=\infty$, then we may take $t_0=\infty$, otherwise $t_0$ is finite.
\end{propn}
\begin{proof} Since $p_t(x,x)=||p_{t/2}(x,\cdot)||_2^2$, we can use the upper bound at (\ref{trandensupper}) to deduce the first inequality. We now claim that
\begin{equation}\label{proofclaim}
h(f_l(r)r)\leq h_l(r)\leq h(r).
\end{equation}
Noting that $V$ is increasing and $f_l(r)\leq 1$ we must have
$$h(f_l(r)r)= f_l(r)rV(f_l(r)r)\leq f_l(r)rV(r)=h_l(r),$$
which is the left hand inequality. To prove the right hand inequality we simply note that $h_l(r)=rf_l(r)V(r)\leq rV(r)=h(r)$. Thus, the claim does indeed hold. If we now define $r$ by $t=h_l(r)$, we have from (\ref{proofclaim}) that $h(f_l(r)r)\leq t\leq h(r)$, and applying $h^{-1}$ to this yields
\begin{equation}\label{secondproofclaim}
f_l(r)h_l^{-1}(t)=f_l(r)r\leq h^{-1}(t)\leq  r.
\end{equation}
With this choice of $r$, the upper bound on the transition density given at (\ref{trandensupper}) transforms to
$$p_t(x,x)\leq c_{1}\frac{h_l^{-1}(t)}{t} \leq c_{1}\frac{h^{-1}(t)}{t}f_l(r)^{-1},$$
where we have applied the left hand inequality of (\ref{secondproofclaim}). To complete the proof we use the right hand inequality of (\ref{secondproofclaim}) to deduce that $f_l(h^{-1}(t))\leq f_l(r)$.
\end{proof}

The aim of the subsequent four lemmas is to deduce bounds on the effective resistance from the centre of a ball to its surface.  We start by proving two lemmas which explains how to move factors in and out of the functions $V$, $f_l$ and $f_u$, and will be used repeatedly later in the article. Next, Lemma \ref{coversize} is a version of the result proved in \cite{BarlowBass}, Lemma 2.7. We show how we can bound the size of a cover of a ball with suitably scaled smaller balls. The result of interest is easily deduced from this result, and appears as Lemma \ref{reslemma}.

{\lem \label{volumegrowthlemma}
Let $\Lambda\geq 1$, then $V(\Lambda r)\leq C_u \Lambda^{\beta_u}V(r)$.}
\begin{proof} Let $n=\lceil \ln \Lambda / \ln 2 \rceil$ and then, using the doubling property of $V$, (\ref{doubling}), we have $V(\Lambda r)\leq C_u^nV(2^{-n}\Lambda r) \leq C_u^{1+ \ln \Lambda / \ln 2} V(r)=C_u \Lambda^{\beta_u}V(r)$.
\end{proof}

\begin{lem} \label{concavebounds} There exist constants $c_{1}, c_{2}$ such that
$$f_l(\lambda r)\geq c_{1} \lambda^b f_l(r),\hspace{20pt}\forall \lambda\in[0,1],r\in[0,R_X),$$
$$f_u(\lambda r)\leq c_{2} \lambda^{-b} f_u(r),\hspace{20pt}\forall \lambda\in[0,1],r\in[0,R_X).$$
\end{lem}
\begin{proof}
We shall only prove the result for the $f_l$. The result for $f_u$ is proved by applying the same argument to $1/f_u$. By assumption, $f_l^{1/b}$ is concave and positive on $[0,r_0]$ and so, for $\lambda\in[0,1]$, $r\in[0,r_0]$,
$$f_l^{1/b}(\lambda r)\geq \lambda f_l^{1/b} (r)+(1-\lambda) f_l^{1/b}(0)\geq \lambda f_l^{1/b}(r).$$
Thus, we have the result for $r\in [0,r_0]$. Now, define $f_l(R_X):=\lim_{r\uparrow R_X} f_l(r)$, which exists in $(0,1]$ by the boundedness and monotonicity of $f_l$. We also have that $f_l(r)\leq f_l (R_X)$, for every $r\in[0,R_X)$. Hence, using the result already established for small $r$, we can deduce, $\forall \lambda\in[0,1],r\in[r_0, R_X)$, that
$$f_l(\lambda r)\geq f_l(\lambda r_0)\geq \frac{f_l(r_0)}{f_l(R_X)} \lambda^b f_l(r),$$
which completes the proof.
\end{proof}

\begin{lem} \label{coversize} Fix $\varepsilon\in(0,1/2]$. For any $r>0, x\in X$, we can find a cover of $B(x,r)$ consisting of fewer than $M$ balls of radius $\varepsilon r$, where
$$M:=c_{1} g(r)^{-1},$$
with $c_{1}$ a constant (depending on $\varepsilon$).
\end{lem}
\begin{proof}
Let $x_1\in B(x,r)$ and choose $x_2,x_3,\dots$ by letting $x_{i+1}$ be any point in $B(x,r)\backslash\cup_{j=1}^i B(x_j,\varepsilon r)$. We do this until we can no longer proceed. Note that we must have the $B(x_i, \varepsilon r/2)$ disjoint and also $\cup_{i=0}^m B(x_i, \varepsilon r/2)\subseteq B(x,r(1+\varepsilon/2))$, where $m$ is the number of balls selected for the cover. It follows that
\begin{eqnarray*}
m V_l (\varepsilon r/2)&\leq&\mu \left(\bigcup_{i=0}^m B(x_i, \varepsilon r/2)\right)\\
&\leq&\mu\left(B(x,r(1+\varepsilon/2))\right)\\
&\leq& V_u (r(1+\varepsilon/2)).
\end{eqnarray*}
Now, by applying Lemma \ref{volumegrowthlemma} and Lemma \ref{concavebounds}, we have that $V_l (\varepsilon r/2) \geq c_{2} V_l(r)$, and also that $V_u(r(1+\varepsilon/2))\leq c_3 V_u(r)$. Hence, we must have $m\leq c_{1} f_u(r)f_l(r)^{-1}=c_{1} g(r)^{-1}$, and so the assertion is proved.
\end{proof}

\begin{lem} \label{reslemma} There is a constant $c_{1}$ such that, for all $r\in [0,R_X/2), x\in X$,
$$c_{1} rg(r)^2\leq R(x,B(x,r)^c)\leq r.$$\end{lem}
\begin{proof} This result may be proved by repeating exactly the same argument as was used in \cite{Kumagai}, Lemma 4.1, with the cover size being determined by Lemma \ref{coversize}.
\end{proof}

We shall now prove bounds on the expected exit time of a resistance ball. For $A\subseteq X$, we shall define
$$T_A:=\inf\{t\geq 0:\:X_t\not\in A\}$$
to be the first exit time from $A$.

\begin{propn} \label{expectedexittimes} There exists a constant $c_{1}$ such that
$$E^{x_0}T_{B(x_0,r)}\geq c_{1} h_l(rg(r)^2),\hspace{20pt}\forall x_0\in X,r\in[0,R_X/2)$$
$$E^{x}T_{B(x_0,r)}\leq  h_u(r),\hspace{20pt}\forall x,x_0\in X,r\in[0,R_X/2).$$
\end{propn}
\begin{proof} Fix $x_0\in X, r\in[0,R_X/2)$ and let $B:=B(x_0,r)$. Then, as in \cite{Kumagai}, Proposition 4.2, it may be deduced that there exists a green kernel $g_B(\cdot,\cdot)$ for the process killed on exiting $B$ that satisfies
\begin{equation}\label{green1}
\mathcal{E}(g_B(x,\cdot),g_B(x,\cdot))=g_B(x,x),\hspace{20pt}
\end{equation}
\begin{equation}\label{green2}
g_B(x,x)=R(x,B^c),\hspace{20pt}
\end{equation}
\begin{equation}\label{green3}
g_B(x,y)\leq g_B(x,x),\hspace{20pt}
\end{equation}
\begin{equation}\label{green4}
E^{x}T_{B}=\int_B g_B(x,y)\mu(dy),
\end{equation}
for all $x,y\in X$.

By the inequality at (\ref{resistance}) for the function $g_B(x_0,\cdot)$, one has that
$$|g_B(x_0,y)-g_B(x_0,x_0)|^2\leq R(x_0,y)\mathcal{E}(g_B(x_0,\cdot),g_B(x_0,\cdot)).$$
By using properties (\ref{green1}) and (\ref{green2}) it follows that
$$\left(1-\frac{g_B(x_0,y)}{g_B(x_0,x_0)}\right)^2\leq\frac{R(x_0,y)}{R(x_0,B^c)}.$$
Using (\ref{green3}) and the lower bound on $R(x,B(x_0,r)^c)$ obtained in Lemma \ref{reslemma}, it may be deduced from the above inequality that for some constant $c_2$, if $y\in B(x_0,c_{2}rg(r)^2)$, then $g_B(x_0,y)\geq \frac{1}{2}g_B(x_0,x_0)$. So, by the representation of $E^{x_0}T_B$ given at (\ref{green4}), we have
\begin{eqnarray*}
E^{x_0}T_{B(x_0,r)}&\geq& \frac{1}{2}R(x_0,B^c)V(x_0,c_{2}rg(r)^2)\\
&\geq &\frac{1}{2}c_{2}rg(r)^2 V_l(c_{2} rg(r)^2)\\
&\geq & c_{3} h_l(rg(r)^2),
\end{eqnarray*}
which proves the lower bound. For the upper bound, we proceed as in \cite{Kumagai}, Proposition 4.2, to obtain for $x\in X$, $E^xT_{B(x_0,r)}\leq rV(x_0,r)$, which immediately implies the result by the volume bounds at (\ref{volumegrowthcondition}).
\end{proof}

We now present a bound on the tail of the exit time distribution which will be sufficient for obtaining the on-diagonal lower bound for the heat kernel. The extra assumption we make on the volume growth for the off-diagonal bounds will also allow us to write this bound in a way that avoids using the rather awkward function $q$. This bound is presented in Proposition \ref{tailprobnice}.

\begin{lem} \label{tailprob} There exist constants $c_{1}, c_{2}, c_{q}$ such that
$$P^x(T_{B(x,r)}\leq t)\leq c_{1}e^{-c_{2} \frac{r}{q^{-1}(t/r)}g(q^{-1}(t/r))^{\gamma_1}},\hspace{20pt}\forall x\in X, r\in(0,R_X), t>0,$$
where $q(r):=c_{q} g(r)^{2\gamma_1}V_u(r)$ and $\gamma_1:=3+2b+2\beta_u$.
\end{lem}
\begin{proof} The proof follows a standard pattern and involves the application of \cite{BarBas}, Lemma 1.1, to strengthen a simple linear bound to an exponential one. We start by deducing the relevant linear bound. By Proposition \ref{expectedexittimes}, we have $E^xT_{B(x,r)}\geq c_{3} h_l(g(r)^2r)$, $\forall x\in X$, $r\in(0,R_X/2)$, from which we may deduce that
\begin{equation}\label{eq1}
E^xT_{B(x,r)}\geq c_{4} r g(r)^{2(1+b+\beta_u)}f_l(r)V(r),
\end{equation}
by using Lemmas \ref{volumegrowthlemma} and \ref{concavebounds}. Furthermore, we may use the Markov property of $(X_t)_{t\geq0}$ to deduce that
\begin{equation}\label{eq2}
E^xT_{B(x,r)}\leq t+E^x1_{T_{B(x,r)}>t}E^{X_t}T_{B(x,r)}.
\end{equation}
Since, $E^{x_0}T_{B(x,r)}\leq  h_u(r)$, comparing (\ref{eq1}) and (\ref{eq2}) yields
$$ c_{4} rg(r)^{2(1+b+\beta_u)}f_l(r)V(r)\leq t+ P^x(T_{B(x,r)}>t) h_u(r),$$
which we may rearrange to obtain
$$P^x(T_{B(x,r)}\leq t)\leq 1- c_{4}g(r)^{3+2b+2\beta_u} +\frac{t}{h_u(r)},$$
our linear bound.

To get the exponential bound requires a kind of chaining argument, which we describe now. Let $n\geq 1$ and define stopping times $\sigma_i$, $i\geq0$ by
$$\sigma_0=0,\hspace{10pt}\sigma_{i+1}=\inf\{s\geq \sigma_i\::\:R(X_s,X_{\sigma_i})\geq r/n\}.$$
Let $\tau_i=\sigma_i - \sigma_{i-1}$, $i\geq 1$. Let $\mathcal{F}_t$ be the filtration generated by $\{X_s:\:s\leq t\}$ and $\mathcal{G}_m=\mathcal{F}_{\sigma_m}$. Our linear bound gives
\begin{eqnarray*}
P^x(\tau_{i+1}\leq t |\mathcal{G}_i)& = & P^{X_{\sigma_i}}(T_{B(X_{\sigma_i},\:r/n)}\leq t)\\
&\leq& 1- c_{4}g(r/n)^{3+2b+2\beta_u} +\frac{t}{h_u(r/n)}\\
&=& p(r/n)+\frac{t}{h_u(r/n)},
\end{eqnarray*}
where $p(r):= 1- c_{4}g(r)^{\gamma_1}\in(\frac{1}{2},1)$ for $r >0$, by reducing $c_4$ if necessary. We have $R(X_{\sigma_i},X_{\sigma_{i+1}})=r/n$ and so $R(X_0, X_t)\leq r$, for every $t\in[0,\sigma_n]$, which means that $\sigma_n=\sum_{i=1}^n\tau_i\leq T_{B(X_0,r)}$. Thus, by \cite{BarBas}, Lemma 1.1,
\begin{eqnarray*}
\ln P^x(T_{B(x,r)}\leq t) & \leq & 2\sqrt{\frac{nt}{ p(r/n) h_u(r/n)}}-n\ln\frac{1}{p(r/n)}\\
&\leq&  4\sqrt{\frac{nt}{ h_u(r/n)}}-c_{4}ng(r/n)^{\gamma_1},
\end{eqnarray*}
where we have used the inequality $\ln (1-x)\leq -x$ for $x\in[0,1]$.

Let $c_{q}=\frac{c_{4}^2}{64}$ so that $q$ is fixed. Now $q$ may be rewritten as
$$q(r)=c_{q}f_l(r)^{2\gamma_1}f_u(r)^{1-2\gamma_1}V(r).$$
Since $2\gamma_1 >0>1-2\gamma_1$, each of the terms in the product is increasing, with $V$ strictly increasing. Thus, $q$ is strictly increasing and $q^{-1}$ may be defined sensibly on the appropriate domain.

We consider first the case $r\geq q^{-1}(t/r)$. Define
\begin{eqnarray*}
n_0&:=& \sup\{n:\:8\sqrt{\frac{nt}{h_u(r/n)}}\leq c_{4}ng(r/n)^{\gamma_1}\}\\
&=& \sup \{n:\:nq^{-1}(t/r)\leq r\}.
\end{eqnarray*}
By assumption, we have $n_0\geq 1$ and because $q^{-1}(t/r)>0$ we must also have $n_0<\infty$. Thus,
$$n_0\leq \frac{r}{q^{-1}(t/r)} < n_0+1,$$
from which it follows that
\begin{eqnarray*}
\ln P^x(T_{B(x,r)}\leq t) & \leq & -c_{5}\left(\frac{r}{q^{-1}(t/r)}-1\right) g(q^{-1}(t/r))^{\gamma_1}\\
&\leq&-c_{5}\left(\frac{r}{q^{-1}(t/r)}\right) g(q^{-1}(t/r))^{\gamma_1}+c_{5},
\end{eqnarray*}
which yields the result in this case. If $r<q^{-1}(t/r)$ then
$$ \frac{r}{q^{-1}(t/r)}g(q^{-1}(t/r))^{\gamma_1}\leq 1,$$
and so we have the result for $r\in(0,R_X/2)$ by choosing $c_{1}$ sufficiently large. This inequality is easily extended to hold for $r\in(0,R_X)$ by  adjusting the constants as necessary.
\end{proof}

We are now ready to prove the on-diagonal lower bound. In the proof, we will use the following observation, which is an immediate consequence of Lemma \ref{concavebounds}: there is a constant $c_{1}$ such that
\begin{equation}\label{gbounds}
g(\lambda r)\geq c_{1} \lambda^{2b}g(r),\hspace{20pt}\forall \lambda\in[0,1],r\in[0,R_X).
\end{equation}

\begin{propn} \label{ondiaglowerhk} There exists a constant $c_{1}$ such that
$$p_t(x,x)\geq c_{1} \frac{h^{-1}(t)}{t} g(h^{-1}(t))^{\theta_1},\hspace{20pt}\forall x\in X, t>0,$$
where $\theta_1$ is chosen to satisfy (\ref{theta1cond}).\end{propn}
\begin{proof} Using Cauchy-Schwarz,
\begin{eqnarray}\label{ondiaglower}
P^x(T_{B(x,r)}> t)^2 &\leq& P^x(X_t\in B(x,r))^2\nonumber\\
&=&\left(\int_{B(x,r)}p_t(x,z)\mu(dz)\right)^2\nonumber\\
&\leq& V(x,r)p_{2t}(x,x)\nonumber\\
&\leq& V_u(r)p_{2t}(x,x).
\end{eqnarray}
We prove the result by choosing a suitable $r$ in this inequality. We shall consider the cases for small and large $t$ separately. Define
$$\gamma_2:=\frac{\theta_1-2\gamma_1}{\beta_u+4b\gamma_1}.$$
We then have $\gamma_1-\gamma_2(1-2b\gamma_1)<0$. Noting that, if $g(r)\not\rightarrow 0$, $f_l(r)$ is bounded below by a strictly positive constant and $f_u(r)$ is bounded above by a finite constant. This means that we have uniform volume doubling and the result is given in \cite{Kumagai}, Proposition 4.3. Thus, we may assume $g(r)\rightarrow 0$ as $r\rightarrow 0$, and for any $c_2,c_3$, we can choose $r'$ such that
$$c_{2} e^{-c_{3} g(r)^{\gamma_1-\gamma_2(1-2b\gamma_1)}}\leq \frac{1}{2},\hspace{20pt}\forall r\leq r'.$$
Now, choose $c_2, c_3$ by Lemma \ref{tailprob}, so that $P^x(T_{B(x,r)}\leq t)$ is bounded above by
\[c_{2}\exp({-c_{3} \frac{r}{q^{-1}(t/r)}g(q^{-1}(t/r))^{\gamma_1}}),\]
choose $r'$ accordingly, and set $t':=r' q(r'g(r')^{\gamma_2})$. For $t\leq t'$ we can find $r\leq r'$ such that $t=r q(rg(r)^{\gamma_2})$ and so, for this choice of $r$ and $t$,
\begin{eqnarray*}
P^x(T_{B(x,r)}\leq t)\leq c_{2}e^{-c_{3} g(r)^{-\gamma_2}g(rg(r)^{\gamma_2})^{\gamma_1}}\leq c_{2} e^{-c_{3} g(r)^{\gamma_1-\gamma_2(1-2b\gamma_1)}}\leq\frac{1}{2},
\end{eqnarray*}
where we have applied the inequality at (\ref{gbounds}) for the second inequality. Thus, (\ref{ondiaglower}) gives that $p_{2t}(x,x)\geq 1/4V_u(r)$. After substituting the definition of $q$ and manipulating we find that
\begin{eqnarray*}
t &=& c_{q} rV_u(rg(r)^{\gamma_2})g(rg(r)^{\gamma_2})^{2\gamma_1}\\
&\geq& c_{4}r g(r)^{\theta_1} V_u(r),
\end{eqnarray*}
and hence $p_{2t}(x,x)\geq c_{4}r g(r)^{\theta_1}/4t$. We also have
$$t\leq c_{q} rV_u(r) g(r)^{2\gamma_1}\leq c_{q}rV(r)g(r)^{2\gamma_1-1}\leq c_{q} h(r)\leq h(c_{5}r),$$
noting that $2\gamma_1>1$ and taking $c_{5}=\max\{1,c_{q}\}$. Consequently, $h^{-1}(t)\leq c_{5}r$ and so
$$p_t(x,x)\geq p_{2t}(x,x)\geq c_{6}\frac{h^{-1}(t)}{t} g(h^{-1}(t))^{\theta_1},$$
using that $p_t(x,x)$ is decreasing in $t$. Hence, we have the bound for $t\leq t'$.

Before proceeding we note that $rg(r)^{-\gamma_1}=O(r^{1-2\varepsilon\gamma_1})\rightarrow 0$, as $r\rightarrow 0$, because, by the bound on $\varepsilon$ and $b$ at (\ref{betaepscondition1}), $2\varepsilon \gamma_1<1$. Therefore, we can choose $\tilde{r}$ less than 1 such that
$$c_{2}e^{-c_{3}\frac{1}{\tilde{r}}g(\tilde{r})^{\gamma_1}}\leq\frac{1}{2}.$$
Choose $t'':=q(\tilde{r})$. Now let $t\geq t''$ and define $r$ by $t=rq(r\tilde{r})$. The right hand side of this equation is increasing and so, because $t$ is bounded below (by $t''$) we can assume that $r$ is bounded below by 1.  Hence, applying Lemma \ref{tailprob} gives
\begin{eqnarray*}
P^x(T_{B(x,r)}\leq t)&\leq& c_{2}e^{-c_{3} \frac{r}{q^{-1}(t/r)}g(q^{-1}(t/r))^{\gamma_1}}\\
&\leq& c_{2}e^{-c_{3} \frac{r}{q^{-1}(t/r)}g(q^{-1}(t/r)/r)^{\gamma_1}}\\
&=&c_{2}e^{-c_{3} \frac{1}{\tilde{r}} g(\tilde{r})^{\gamma_1}}\\
&\leq&\frac{1}{2}.
\end{eqnarray*}
Hence, we also have $p_{2t}(x,x)\geq 1/4V_u(r)$ in this case, by (\ref{ondiaglower}). By bounding $t$ in a similar way to the case $t\leq t'$ it may be deduced from this that
$$p_t(x,x)\geq  c_{7}\frac{h^{-1}(t)}{t} g(h^{-1}(t))^{2\gamma_1},$$
and so we have the bound in this case, because $2\gamma_1\leq\theta_1$. Finally, for $t\in(t',t'')$ we may obtain the result by choosing $c_{1}$ small enough.
\end{proof}

We conclude this section by proving the fluctuation results of Theorem \ref{fluctres}.

\begin{proof}[of Theorem \ref{fluctres}]  The left hand inequality of (\ref{infres}) and the right hand inequality of (\ref{supres}) are immediate corollaries of Propositions \ref{ondiagupper} and \ref{ondiaglowerhk}. We now prove the right hand inequality of (\ref{infres}). As at (\ref{trandensupper}), we repeat the argument of \cite{Kumagai} to obtain
\begin{equation}\label{trandensupper2}
p_{2rV(x,r)}(x,x)\leq \frac{2}{V(x,r)},\hspace{20pt}\forall x\in X,r\in[0,R_X).
\end{equation}
Hence, because $p_t(x,x)$ is decreasing in $t$, this means that
$$\inf_{x\in X} p_{2h_u(r)}(x,x)\leq \inf_{x\in X} p_{2rV(x,r)}(x,x)\leq \frac{2}{\sup_{x\in X}V(x,r)} \leq \frac{c_{1}}{V_u(r)},$$
for all $r\in[0,R_X)$, where we use the assumption at (\ref{supcond}) for the final inequality. Setting $r=h_u^{-1}(t/2)$, we obtain $\inf_{x\in X}p_t(x,x)\leq c_{2} h_u^{-1}(t)/t$, which gives the result.

It remains to prove the left hand inequality of (\ref{supres}). The majority of the proof of this consists of repeating arguments that are almost identical to those we have seen already, and so we omit many of the details here. By the assumption at (\ref{infcond}), we can find a sequence $(x_n, r_n)_{n\in\mathbb{N}}$ such that $x_n\in X$, $r_n\rightarrow 0$ and $V(x_n, r_n)\leq c_{3} V_l(r_n)$. By proceeding similarly to the proofs of Lemmas \ref{coversize} and \ref{reslemma}, it may be deduced that
$$c_{4}r_n\leq R(x_n, B(x_n, r_n)^c)\leq r_n,\hspace{20pt}\forall n\in \mathbb{N}.$$
Using this result, by following the argument of Proposition \ref{expectedexittimes}, we find that
$$E^{x_n}T_{B(x_n,r_n)}\geq c_{5} h_l(r_n)\hspace{20pt}\forall n\in \mathbb{N},$$
and
$$E^{x}T_{B(x_n,r_n)}\leq  c_{3} h_l(r_n),\hspace{20pt}\forall x\in X, n \in\mathbb{N}.$$
Thus, by utilising the Markov property of $(X_t)_{t\geq 0}$ as at (\ref{eq2}), it follows that
$$P^{x_n}(T_{B(x_n, r_n)}\leq t) \leq 1 - \frac{c_{5}}{c_{3}}+\frac{t}{c_{3}h_l(r_n)},$$
and in particular
$$P^{x_n}\left(T_{B(x_n, r_n)}\leq \frac{c_{5}}{2}h_l(r_n) \right)\leq 1 - \frac{c_{5}}{2c_{3}}<1,\hspace{20pt}\forall n\in \mathbb{N}.$$
The Cauchy-Schwarz inequality at (\ref{ondiaglower}) applied to $x_n$, $r_n$ and $t_n=c_{5}h_l(r_n)/2$ will then imply that
$$\sup_{x\in X}p_{t_n}(x,x)\geq p_{t_n}(x_n,x_n)\geq \frac{c_{6}}{V(x_n,r_n)}\geq \frac{c_{6}}{c_{3}V_l(r_n)}\geq\frac{c_{7} h_l^{-1}(t_n)}{t_n}.$$
Noting that $t_n\rightarrow 0$, this completes the proof.
\end{proof}

\section{Proof of off-diagonal heat kernel bounds}\label{offdiagonal}

Throughout this section, we shall be assuming the extra anti-doubling condition on the volume growth, (\ref{volumegrowthlower}), and the tighter upper bounds on $b$ and $\varepsilon$, (\ref{betaepscondition2}), that were stated in Section \ref{offdiagstate}. These allow us to obtain the off-diagonal estimates stated there. We start by presenting a counterpart to Lemma \ref{volumegrowthlemma} for small $\lambda$, which the extra volume growth condition implies.

\begin{lem} \label{volumegrowthlemma2} Let $\lambda\leq 1$, then $V(\lambda r)\leq C_l \lambda^{\beta_l}V(r)$, for every $r\leq R_X'$.
\end{lem}
\begin{proof} This follows a similar argument to the proof of Lemma \ref{volumegrowthlemma}.
\end{proof}

As is usually the case in situations similar to this, the off-diagonal upper bound is relatively straightforward to obtain from the upper bounds for the on-diagonal part of the heat kernel and the tail of the exit time distribution of resistance balls. However, before proceeding with the proof of the off-diagonal upper bound, it will be useful to write the result of Lemma \ref{tailprob} in a slightly clearer form.

\begin{propn} \label{tailprobnice} If $R_X'=\infty$, let $t_0=\infty$, otherwise fix $t_0\in(0,\infty)$. Then there exist constants $c_{1}, c_{2}$ such that
$$P^x(T_{B(x,r)}\leq t)\leq c_{1}e^{-c_{2} \frac{r}{V^{-1}(t/r)}g(V^{-1}(t/r))^{\theta_3}},\hspace{8pt}\forall x\in X, r\in(0,R_X), t\in(0,t_0).$$
\end{propn}
\begin{proof} In Lemma \ref{tailprob} we obtained a bound for the relevant probability in terms of the function $q^{-1}$. To establish this claim we use Lemma \ref{volumegrowthlemma2} to compare $q^{-1}$ to functions of $V^{-1}$ and $g$ only. Recall $q(r)=c_{q}g(r)^{2\gamma_1}V_u(r)$, and so for $r\leq R_X'$, we have $q(r)\geq V(c_{3} r g(r)^{2\gamma_1/\beta_l})$, for some constant $c_{3}$. Thus,
\begin{equation}\label{eq3}
V^{-1}(t/r)\geq c_{3} q^{-1}(t/r)g(q^{-1}(t/r))^{2\gamma_1/\beta_l},
\end{equation}
for $t/r\leq q(R_X')$. We also have the following upper bound on $q$
$$q(r)\leq c_{q} V(r) g(r)^{2\gamma_1-1}\leq c_{q} V(r) \leq V(c_{4} r),$$
where $c_{4}=\max\{(c_{q}C_l)^{1/\beta_l},1\}$, which holds whenever $c_{4}r\leq R_X'$. Thus,
\begin{equation}\label{eq4}
V^{-1}(t/r)\leq c_{4} q^{-1}(t/r),
\end{equation}
for $t/r \leq q(c_{4}^{-1}R_X')$. Combining the bounds at (\ref{eq3}) and (\ref{eq4}) we find that
$$\frac{r}{q^{-1}(t/r)}g(q^{-1}(t/r))^{\gamma_1}\geq c_{5}\frac{r}{V^{-1}(t/r)}g(V^{-1}(t/r))^{\gamma_1(1+2\beta_l^{-1})}$$
for all $t/r \leq q(c_{6}R_X')$, where $c_{6}:=\min\{c_{4}^{-1},1\}$. Thus, we have the result when $R_X'=\infty$. Assume now $R_X'<\infty$ and fix $t_0<\infty$. The previous equation gives us the result when $t/r \leq q(c_{6}R_X')$ and so we can assume that this does not hold. Hence,
$$\frac{r}{V^{-1}(t/r)}g(V^{-1}(t/r))^{\theta_3}\leq\frac{t_0}{q(c_{6}R_X')V^{-1}(q(c_{6}R_X'))},\hspace{20pt}\forall t<t_0,$$
and so the result will hold on choosing $c_{1}$ suitably large.
\end{proof}

\begin{propn} \label{offdiagupperprop}  We can find a $t_0>0$ such that the following holds: there exists $c_{1}, c_{2}$ such that, if $x,y\in X$, $t\in(0,t_0)$,
$$p_t(x,y)\leq  c_{1}\frac{h^{-1}(t)}{t}f_l(h^{-1}(t))^{-1}e^{-c_{2} \frac{R}{V^{-1}(t/R)}g(V^{-1}(t/R))^{\theta_3}},$$
where $R=R(x,y)$. If $R_X'=\infty$, then we can take $t_0=\infty$, otherwise $t_0\in(0,\infty)$.
\end{propn}
\begin{proof} Once we have the on-diagonal bound, Lemma \ref{ondiagupper}, and the exponential bound for the exit time distribution, Lemma \ref{tailprobnice}, the proof is standard, see \cite{Barlow}, Theorem 3.11.
\end{proof}

We now start to work towards the full lower bound. We start by deducing the near diagonal result using a modulus of continuity argument. This is the extension of the result obtained by Kumagai in the uniform volume doubling case, as stated at (\ref{kumagairesultlower}).

\begin{lem}\label{nle}  There exist constants $c_{1},c_{2}$ such that, whenever $x,y\in X$ satisfy
$$R(x,y)\leq c_{1} h^{-1}(t) g(h^{-1}(t))^{\theta_1},$$
we have
$$p_t(x,y)\geq c_{2} \frac{h^{-1}(t)}{t} g(h^{-1}(t))^{\theta_1},\hspace{20pt}\forall t>0.$$
\end{lem}
\begin{proof} The proof is again standard. For any $x\in X$, $t>0$, it is known that the transition density satisfies $\mathcal{E}(p_t(x,\cdot),p_t(x,\cdot))\leq p_t(x,x)/t$. For a proof, see \cite{Barlow}, Proposition 4.16. In conjunction with the inequality at (\ref{resistance}), we obtain from this that
$$|p_t(x,x)-p_t(x,y)|^2\leq R(x,y) \mathcal{E}(p_t(x,\cdot),p_t(x,\cdot))\leq R(x,y)\frac{p_t(x,x)}{t}.$$
Thus,
\begin{eqnarray*}
p_t(x,y)&\geq& p_t(x,x)-|p_t(x,x)-p_t(x,y)|\\
&\geq &p_t(x,x)\left( 1-\sqrt{\frac{R(x,y)}{tp_t(x,x)}}\right)\\
&\geq&\frac{1}{2}p_t(x,x),
\end{eqnarray*}
whenever $4R(x,y)\leq tp_t(x,x)$. Consequently, the result may be obtained by applying the on-diagonal lower bound obtained in Proposition \ref{ondiaglowerhk}.
\end{proof}

To prove the full lower bound we shall assume the chaining condition as defined in Section \ref{offdiagstate}. We shall use the standard chaining argument to extend the near diagonal lower bound to the full bound. The main complication caused by the perturbations is in choosing a suitable number of pieces to break the path into. The aim of the following lemma is to check that the number that we do choose is sensibly defined.

\begin{lem}\label{pathparts} Fix $c_{1}$. Let $x,y\in X$ and $t>0$. If we define $N=N(x,y,t)$ by
$$N:=\inf\{n\in \mathbb{N}:\:\frac{R(x,y)}{n}\leq c_{1}h^{-1}(t/n)g(h^{-1}(t/n))^{\theta_1}\},$$
then $N$ is well-defined and finite for each pair $x,y\in X$.
\end{lem}
\begin{proof} Note first that $h^{-1}(t)/t=1/V(h^{-1}(t))$, so we can rewrite $N$ as
$$N=\inf\{n\in \mathbb{N}:\:\frac{R(x,y)}{t}\leq \frac{c_{1}}{V(h^{-1}(t/n))}g(h^{-1}(t/n))^{\theta_1}\}.$$
It is clear that $h^{-1}(t/n)\rightarrow 0$ as $n\rightarrow \infty$ and so, to prove the lemma, it suffices to show that $V(r)g(r)^{-\theta_1}\rightarrow 0$ as $r\rightarrow 0$. By Lemma \ref{volumegrowthlemma2} we have
$$V(r)g(r)^{-\theta_1}\leq C_l V(rg(r)^{-\theta_1/\beta_l}),$$
for $rg(r)^{-\theta_1/\beta_l}\leq R_X'$. We note that, using the assumptions of Sections \ref{framework} and \ref{offdiagstate}, we have $rg(r)^{-\theta_1/\beta_l}=O(r^{1-2\varepsilon\theta_1/\beta_l})\rightarrow 0$ as $r\rightarrow 0$, and so the result does indeed hold.
\end{proof}

We are now ready to state and prove the full lower bound. We now assume that the chaining condition, $(CC)$, holds.

\begin{propn} \label{offdiaglowerprop} There exist constants $t_0>0$ and $c_{1}, c_{2}$ such that, if $x,y\in X$, $t\in(0,t_0)$,
$$p_t(x,y)\geq  c_{1}\frac{h^{-1}(t)}{t} g(h^{-1}(t))^{\theta_1}e^{-c_{2} \frac{R}{V^{-1}(t/R)}g(V^{-1}(t/R))^{-\theta_2}},$$
where $R=R(x,y)$. If $R_X'=\infty$ then we may take $t_0=\infty$, otherwise $t_0$ will be finite.
\end{propn}
\begin{proof} Let $x,y\in X$ and $R=R(x,y)$. Now, there exists a constant $c_3$ such that, if $R\leq c_{3} h^{-1}(t) g(h^{-1}(t))^{\theta_1}$ then we have the result by Lemma \ref{nle} immediately. Thus, we need only consider the case $R> c_{3} h^{-1}(t) g(h^{-1}(t))^{\theta_1}$. We shall use a standard chaining argument using the previous lemma to select the length of the path we shall use. Define
$$N=\inf\{n\in \mathbb{N}:\:\frac{R(x,y)}{n}\leq \frac{c_{3}}{3c_{4}}h^{-1}(t/n)g(h^{-1}(t/n))^{\theta_1}\},$$
where $c_4$ is the constant that appears in the chaining condition. Lemma \ref{pathparts} and the assumption on $R$ means that $N\in(1,\infty)$. By the chaining condition we can find a path $x=x_0,x_1,\dots,x_N=y$ such that
$$R(x_{i-1},x_i)\leq\frac{c_{4}R}{N},\hspace{20pt}i=1,\dots,N.$$
If we set $\delta= c_{3}h^{-1}(t/N)g(h^{-1}(t/N))^{\theta_1}$, then by the definition of $N$, this inequality implies that $R(x_{i-1},x_i)\leq \delta/3$, for $i=1,\dots,N$. Thus, if $z_i\in B(x_i, \delta/3)$, we have
$$R(z_{i-1},z_i)\leq\delta,\hspace{20pt}i=1,\dots,N,$$
and so we may apply the near diagonal estimate to obtain
\begin{equation}\label{eq5}
p_{t/N}(z_{i-1},z_i)\geq c_{5} \frac{N h^{-1}(t/N)}{t} g(h^{-1}(t/N))^{\theta_1}.
\end{equation}

This is the first ingredient that we shall require to apply the chaining argument. The other is a lower bound on the measures of the balls $B(x_i, \delta/3)$. Using the assumption (\ref{volumegrowthcondition}),
\begin{eqnarray}
V(x_i, \delta/3)&\geq& V_l(\delta/3)\nonumber\\
&=& V_l(c_{3}h^{-1}(t/N)g(h^{-1}(t/N))^{\theta_1}/3)\nonumber\\
&\geq& c_{6} V(h^{-1}(t/N))g(h^{-1}(t/N))^{1+\theta_1(b+\beta_u)},\label{eq6}
\end{eqnarray}
where we have applied Lemmas \ref{volumegrowthlemma} and \ref{concavebounds} to obtain the second inequality.

By using the Chapman-Kolmogorov equation for the transition densities of the process $X$ we obtain the following chaining inequality
$$p_t(x,y)\geq \int_{B(x_1,\:\delta/3)}\mu(dz_1)\dots\int_{B(x_{N-1},\:\delta/3)}\mu(dz_{N-1})\prod_{i=1}^N p_{t/N}(z_{i-1},z_i).$$
If we then combine this with the bounds at (\ref{eq5}) and (\ref{eq6}) we obtain
\begin{eqnarray}
p_t(x,y)&\geq& c_{5} \frac{N h^{-1}(t/N)}{t} g(h^{-1}(t/N))^{\theta_1}\nonumber\\
& &\hspace{20pt}\times \left(c_{5} c_{6} g(h^{-1}(t/N))^{1+\theta_1(b+\beta_u+1)}\right)^{N-1},\nonumber
\end{eqnarray}
where we have used the identity $h^{-1}(t)/t=1/V(h^{-1}(t))$. The definition of $N$ and the assumption that $R> c_{3} h^{-1}(t) g(h^{-1}(t))^{\theta_1}$ may be combined to give
$$\frac{N h^{-1}(t/N)}{t} g(h^{-1}(t/N))^{\theta_1}\geq 3c_{4}\frac{h^{-1}(t)}{t} g(h^{-1}(t))^{\theta_1},$$
yielding
\begin{equation}\label{eq8}
p_t(x,y)\geq c_{7}\frac{h^{-1}(t)}{t} g(h^{-1}(t))^{\theta_1}e^{-c_{8}N(1-c_{9}\ln g (h^{-1}(t/N)))}.
\end{equation}

To complete the argument we look for bounds on the terms involving  $N$. Since we know that $N>1$ we can deduce, because $h^{-1}(t)$ is increasing,
$$\frac{R}{N}=\frac{R}{N-1}\frac{N-1}{N}\geq \frac{c_{3}}{6c_{4}}h^{-1}(t/N)g(h^{-1}(t/N))^{\theta_1},$$
which we can rewrite as
\begin{equation}\label{eq9}
\frac{t}{R}\leq c_{10} V(h^{-1}(t/N))g(h^{-1}(t/N))^{-\theta_1}.
\end{equation}
Since $g(r)^{-1}=O(r^{-2\varepsilon})$ and $2\varepsilon\theta_1/\beta_l<1$, we can find a $t_0>0$ such that
\begin{equation}\label{eq10}
c_{11}h^{-1}(t)g(h^{-1}(t))^{-\theta_1/\beta_l}\leq R_X',\hspace{20pt}\forall t<t_0,
\end{equation}
where $c_{11}:=\max\{C_l^{1/\beta_l},(C_lc_{10})^{1/\beta_l}\}$. Note that if $R_X'=\infty$ we may take $t_0=\infty$. Clearly this also implies that $c_{10}h^{-1}(t/N)g(h^{-1}(t/N))^{-\theta_1/\beta_l}\leq R_X'$, for $t<t_0$. Thus, applying Lemma \ref{volumegrowthlemma2} to (\ref{eq9}) gives
\begin{equation}\label{eq11}
\frac{t}{R} \leq V(c_{11}h^{-1}(t/N)g(h^{-1}(t/N))^{-\theta_1/\beta_l}),
\end{equation}
and so
\begin{equation}\label{thisisa}
g(\tilde{V})\leq c_{12}g(h^{-1}(t/N))^{1-2b\theta_1/\beta_l},
\end{equation}
where $\tilde{V}:=V^{-1}(t/R)$. By using this in (\ref{eq9}) we find
$$V(h^{-1}(t/N))\geq c_{13} \frac{t}{R} g(\tilde{V})^{\theta_1\beta_l/(\beta_l-2b\theta_1)},$$
and moreover, equations (\ref{eq10}) and (\ref{eq11}) imply that $\tilde{V}\leq R_X'$. These facts allow us to deduce, after some manipulation and the use of Lemma \ref{volumegrowthlemma2}, that
$$h^{-1}(t/N)\geq V^{-1}(c_{13} \frac{t}{R} g(\tilde{V})^{\theta_1\beta_l/(\beta_l-2b\theta_1)})\geq c_{14} \tilde{V} g(\tilde{V})^{\theta_1/(\beta_l-2b\theta_1)}.$$
Consequently,
$$\frac{t}{N}\geq c_{13}c_{14} \frac{t}{R}\tilde{V} g(\tilde{V})^{\theta_1(\beta_l+1)/(\beta_l-2b\theta_1)},$$
which is equivalent to
\begin{equation}\label{eq13}
 N\leq c_{15} \frac{R}{\tilde{V}}g(\tilde{V})^{-\theta_1(\beta_l+1)/(\beta_l-2b\theta_1)}.
\end{equation}
Substituting the bounds of (\ref{thisisa}) and (\ref{eq13}) into the lower bound we established at (\ref{eq8}) yields
\begin{eqnarray*}
p_t(x,y) &\geq& c_{7}\frac{h^{-1}(t)}{t} g(h^{-1}(t))^{\theta_1}\\
& &\hspace{20pt}\times e^{-c_{16}\frac{R}{\tilde{V}}g(\tilde{V})^{-\theta_1(\beta_l+1)/(\beta_l-2b\theta_1)}(1-c_{17}\ln g (\tilde{V}))},\\
&\geq& c_{7}\frac{h^{-1}(t)}{t} g(h^{-1}(t))^{\theta_1}e^{-c_{18}\frac{R}{\tilde{V}}g(\tilde{V})^{-\theta_2}},\hspace{60pt}\forall t<t_0,
\end{eqnarray*}
which is the desired result.
\end{proof}

\section{Local fluctuations}\label{local}

In \cite{HamKum}, Hambly and Kumagai demonstrated that for certain random recursive Sierpinski gaskets, as well as spatial fluctuations, the heat kernel will undergo fluctuations in time $\mu$-almost-everywhere in $X$. In this section, we look to generalise this result by showing that these local fluctuations in the heat kernel result from local fluctuations in the measure.

Again, we shall be working with the measure-metric space $(X,d,\mu)$ and the volume function $V$. We shall denote the local fluctuations by $\tilde{f}_l$ and $\tilde{f}_u$ and assume that these satisfy the same properties as did $f_l$ and $f_u$, respectively. In fact, the results proved here may be obtained using slightly weaker assumptions, but we omit these for brevity. We shall use $\tilde{V}_l(r)$, $\tilde{V}_u(r)$, $\tilde{h}_l(r)$ and $\tilde{h}_u(r)$ to notate $\tilde{f}_l(r)V(r)$, $\tilde{f}_u(r) V(r)$, $r\tilde{V}_l(r)$ and $r\tilde{V}_u(r)$, respectively.

For the following theorem, we make only point-wise assumptions on the volume growth. Because of this, we cannot establish a lower bound on $R(x, B(x,r)^c)$. As we need some kind of global control on this, we simply take as an assumption that it is bounded below by a multiple of $r$. Note that this is a stricter condition than the one established at Lemma \ref{reslemma} when we had global bounds on the measure.

\begin{thm} \label{localres} If
$$0<\liminf_{r\rightarrow0} \frac{V(x,r)}{\tilde{V}_l(r)}<\infty,\hspace{20pt}0<\limsup_{r\rightarrow0} \frac{V(x,r)}{\tilde{V}_u(r)},$$
and
\begin{equation}\label{rescond}
0<\liminf_{r\rightarrow 0} \frac{R(x,B(x,r)^c)}{r}
\end{equation}
for $\mu$-almost-every $x\in X$; then
\begin{equation}\label{localinfres}
\liminf_{t\rightarrow0}\frac{tp_t(x,x)}{\tilde{h}_u^{-1}(t)}<\infty,
\end{equation}
and
\begin{equation}\label{localsupres}
0<\limsup_{t\rightarrow0} \frac{tp_t(x,x)}{\tilde{h}_l^{-1}(t)}<\infty,
\end{equation}
for $\mu$-almost-every $x\in X$.
\end{thm}
\begin{proof} The bound at (\ref{localinfres}) is proved by applying the inequality at (\ref{trandensupper2}) in exactly the same way as in the proof of the corresponding global bound. A similar argument is also used to prove the upper bound of (\ref{localsupres}).

The assumption on $R$ at (\ref{rescond}) allows us to deduce that for $\mu$-almost-every $x\in X$, there exists a sequence $r_n\rightarrow 0$ such that
$$E^{x}T_{B(x,r_n)}\geq c_{1} \tilde{h}_l(r_n)\hspace{20pt}\forall n\in \mathbb{N},$$
and
$$E^{y}T_{B(x,r_n)}\leq  c_{2} \tilde{h}_l(r_n),\hspace{20pt}\forall y\in X, n \in\mathbb{N},$$
by following the argument of Proposition \ref{expectedexittimes}. The result at (\ref{localsupres}) follows from this by applying the Markov property of our process and the Cauchy-Schwarz inequality as we did for the analogous global bound.
\end{proof}
\begin{rem} Using the techniques of this article, it is not enough to assume that
\[\limsup_{r\rightarrow0}(V(x,r)/\tilde{V}_u(r))<\infty\]
to establish a lower bound on $p_t(x,x)$ that holds for all small $t$. The problem arises because we are unable to emulate the chaining argument that was used in Proposition \ref{tailprob} to establish an exponential tail for the distribution of the exit time from a ball.
\end{rem}
\begin{rem} Similar to the remark made after Theorem \ref{fluctres}, we note there are non-trivial local fluctuations in the measure if and only if $\tilde{V}_u(r)/\tilde{V}_l(r)\rightarrow \infty$ as $r\rightarrow 0$. This is equivalent to $\tilde{h}_l^{-1}(t)/\tilde{h}_u^{-1}(t)\rightarrow \infty$ as $t\rightarrow0$, which implies that there are non-trivial local fluctuations in the heat kernel.\end{rem}

\section{Examples}\label{examples}

In this section, to illustrate the results, we look at two specific examples of correction terms and present the conclusions for two particular random sets. In the Sections \ref{polycorr} and \ref{logcorr} we shall take $V(r)=r^{\alpha}$ for some $\alpha>0$, so that $\beta_l=\beta_u=\alpha$. For simplicity, we assume that $R_X=\infty$ and the chaining condition holds. In this case, we have $V^{-1}(t)=t^{1/\alpha}$ and $h^{-1}(t)=t^{1/(1+\alpha)}$. Furthermore, in the case of uniform volume growth with volume doubling, we can use the results of Kumagai to show that
$$c_{1}t^{-\frac{\alpha}{\alpha+1}}e^{-c_{2}\left(\frac{R^{\alpha+1}}{t}\right)^{1/\alpha}} \leq p_t(x,y)\leq c_{3}t^{-\frac{\alpha}{\alpha+1}}e^{-c_{4}\left(\frac{R^{\alpha+1}}{t}\right)^{1/\alpha}}$$
for this choice of volume growth function.

\subsection{Polynomial corrections}\label{polycorr}

We first discuss the case of arbitrary polynomial corrections. We shall assume that given $\delta>0$, there exist constants $c_{1}, c_{2}$ such that
$$c_{1}r^\alpha(r^\delta\wedge 1)\leq V(x,r)\leq c_{2}r^\alpha(r^{-\delta}\vee 1),\hspace{20pt}\forall x\in X, r\geq0,$$
so that $f_l(r)=c_{1}(r^\delta\wedge 1)$ and $f_u(r)=c_{2}(r^{-\delta}\vee 1)$. If we set $\varepsilon = b = \delta$, then $f_l$, $f_u$ satisfy the conditions for the full bounds when $\delta<\alpha/8(3+\alpha)^2$. We can then also choose
$$\theta_1=4(2+\alpha)^2,\hspace{10pt}\theta_2=\frac{4(2+\alpha)^3}{\alpha-8\delta(2+\alpha)^2},\hspace{10pt}\theta_3=(3+2\delta+2\alpha)(1+2\alpha^{-1}),$$
and apply Theorem \ref{fullbounds} to obtain that
$$c_{3}t^{-\frac{\alpha-2\delta\theta_1}{\alpha+1}}e^{-c_{4}\left(\frac{R^{1+\alpha-2\delta\theta_2}}{t^{1-2\delta\theta_2}}\right)^{1/\alpha}} \leq p_t(x,y)\leq c_{5}t^{-\frac{\alpha+\delta}{\alpha+1}}e^{-c_{6}\left(\frac{R^{1+\alpha+2\delta\theta_3}}{t^{1+2\delta\theta_3}}\right)^{1/\alpha}},$$
for appropriate $t,x,y$. We note that $\delta, 2\delta\theta_1, 2\delta\theta_2, 2\delta\theta_3 \rightarrow 0$ as $\delta\rightarrow 0$, and so, by taking $\delta$ small enough, we can write down bounds with arbitrarily small polynomial correction terms.

\subsection{Logarithmic corrections}\label{logcorr}

Assume now that
\begin{equation}\label{loginfcond}
0<\liminf_{r\rightarrow0}\inf_{x\in X} \frac{V(x,r)}{V(r)(\ln r^{-1})^{-a_1}} \leq \limsup_{r\rightarrow0}\inf_{x\in X} \frac{V(x,r)}{V(r)(\ln r^{-1})^{-a_1}}<\infty,
\end{equation}
and
\begin{equation}\label{logsupcond}
0<\liminf_{r\rightarrow0}\sup_{x\in X} \frac{V(x,r)}{V(r)(\ln r^{-1})^{a_2}}\leq \limsup_{r\rightarrow0}\sup_{x\in X} \frac{V(x,r)}{V(r)(\ln r^{-1})^{a_2}}<\infty;
\end{equation}
for some $a_1, a_2\in(0,\infty)$. As we noted in the introduction, this is an example that arises naturally in the random recursive fractal setting. We have $f_l(r)=c_{1}(\ln r^{-1})^{-a_1}$ and $f_u(r)= c_{2}(\ln r^{-1})^{a_2}$, which satisfy the conditions for any $\varepsilon, b>0$. Thus, by applying Theorem \ref{fullbounds} we can deduce full heat kernel bounds with $\theta_1, \theta_2,\theta_3$ arbitrarily close to the lower bounds of
$$\theta_1>(3+2\alpha)(2+\alpha),\:\theta_2>\frac{(3+2\alpha)(2+\alpha)(1+\alpha)}{\alpha},\:\theta_3>(3+2\alpha)(1+2\alpha^{-1}),$$
as long as $\theta_1, \theta_2$ satisfy (\ref{theta2cond}). Thus, our results show that the correction terms in the heat kernel will be of logarithmic order. In fact, because we know the functions explicitly, by repeating the same arguments as in previous sections more carefully, we can improve these exponents. Theorem \ref{fluctres} allows us to deduce that the on-diagonal part of the heat kernel satisfies
\begin{equation}\label{eqeq1}
0<\liminf_{t\rightarrow0}\inf_{x\in X} \frac{p_t(x,x)}{t^{-\frac{\alpha}{\alpha+1}}(\ln t^{-1})^{-\frac{\alpha(2\alpha+3)(\alpha+2)a_0+a_2}{\alpha+1}}},\hspace{20pt}
\end{equation}
$$\hspace{20pt}\limsup_{t\rightarrow0}\inf_{x\in X} \frac{p_t(x,x)}{t^{-\frac{\alpha}{\alpha+1}}(\ln t^{-1})^{-\frac{a_2}{\alpha+1}}}<\infty$$
and
$$0<\liminf_{t\rightarrow0}\sup_{x\in X} \frac{p_t(x,x)}{t^{-\frac{\alpha}{\alpha+1}}(\ln t^{-1})^{-\frac{a_1}{\alpha+1}}}\leq\limsup_{t\rightarrow0}\sup_{x\in X} \frac{p_t(x,x)}{t^{-\frac{\alpha}{\alpha+1}}(\ln t^{-1})^{-\frac{a_1}{\alpha+1}}}<\infty,$$
where $a_0:=a_1+a_2$, and we have sharpened the exponent $\theta_1$.

\subsection{Random recursive Sierpinski gaskets}\label{gaskets}

We now compare the above results for logarithmic corrections to those that are known to hold for the random recursive Sierpinski gasket described in \cite{HamJon}. The gasket does not satisfy the chaining condition, but since we do not need this for the on-diagonal results, our results still apply. As noted in the introduction, for this gasket, the results of \cite{HamJon} may be adapted to show there are fluctuations in the measure of resistance balls of the type described at (\ref{loginfcond}) and (\ref{logsupcond}) for some $a_1, a_2>0$.

Our results for the asymptotics of $\sup_{x\in X} p_t(x,x)$ are tight and agree with those found in \cite{HamKum} by Hambly and Kumagai for these random sets. We also have that the upper bound on $\inf_{x\in X} p_t(x,x)$ agrees with the result proved there. We observe that the heat kernel bounds obtained for this gasket in \cite{Hambly} imply that
$$0<\liminf_{t\rightarrow0}\inf_{x\in X} \frac{p_t(x,x)}{t^{-\frac{\alpha}{\alpha+1}}(\ln t^{-1})^{-\frac{\alpha a_0+a_2}{\alpha+1}}},$$
and so the lower bound at (\ref{eqeq1}) has a strictly worse exponent than is optimal. The main reason for this is that, because we have not taken into account the structure of the space, our lower bound on $R(x, B(x,r)^c)$ is not tight. Using results of \cite{Hambly}, we deduce that $c_{1}r\leq R(x, B(x,r)^c)$ for this gasket, whereas Lemma \ref{reslemma} only allows us to obtain $c_{2}r(\ln r^{-1})^{-2a_0}\leq R(x, B(x,r)^c)$.

We note that, because $c_{1}r\leq R(x, B(x,r)^c)$, the local measure results proved in \cite{HamJon} also may be adapted to enable us to apply Theorem \ref{localres} to demonstrate there are fluctuations in time for the heat kernel on this gasket with $\tilde{f}_l(r)=c_{3}(\ln\ln r^{-1})^{-a_1}$ and $\tilde{f}_u(r)=c_{4}(\ln\ln r^{-1})^{a_2}$. That fluctuations of this kind exist was first proved in \cite{HamKum}, and it may be readily observed that the bounds of Theorem \ref{localres} agree with the corresponding results of that paper. Finally, as was noted in the remark following Theorem \ref{localres}, we are unable to establish a local lower bound for $p_t(x,x)$ for small $t$ in the general case, whereas, by taking into account the specific structure of the sets involved, Hambly and Kumagai are able to do so in this particular example.

\subsection{Continuum random tree}

The continuum random tree is a significant example of a random dendrite, with connections to branching processes, graph theory and super-processes, see \cite{Aldous2} for an overview. Possessing a natural shortest path metric, it fits naturally into the resistance form framework, and so the problem of establishing good heat kernel bounds reduces to that of finding good measure bounds for the set. This is the aim of \cite{Croydoncrt}, in which logarithmic global measure fluctuations about a leading order $r^2$ term are demonstrated. Recall, for a dendrite the resistance metric is actual identical to the original one if this is a shortest path metric, and so the conclusions drawn there may be taken to be for resistance balls. Thus, the results of Theorem \ref{fluctres} and Theorem \ref{fullbounds} both apply, yielding full heat kernel estimates and global logarithmic fluctuations in the on-diagonal part of the heat kernel about the leading term of $t^{-2/3}$. Furthermore, as for the self-similar gaskets of Section \ref{gaskets}, the local fluctuations in the measure and heat kernel are shown to be of log-logarithmic order. These results are closely related to those discussed in the introduction for the incipient infinite cluster of critical percolation on the binary tree. In particular, this random graph, when rescaled, converges in distribution to a random set known as the self-similar continuum random tree, which is made up of a Poissonian collection of independent continuum random trees (see \cite{Aldous2}).

\def\cprime{$'$}
\providecommand{\bysame}{\leavevmode\hbox to3em{\hrulefill}\thinspace}
\providecommand{\MR}{\relax\ifhmode\unskip\space\fi MR }
\providecommand{\MRhref}[2]{%
  \href{http://www.ams.org/mathscinet-getitem?mr=#1}{#2}
}
\providecommand{\href}[2]{#2}

\end{document}